\documentclass[12pt]{amsart}
\usepackage[nobysame,abbrev,alphabetic]{amsrefs}
\usepackage{amssymb}
\usepackage[all]{xy}

\DeclareMathOperator{\Ann}{Ann}

\DeclareMathOperator{\GL}{GL}

\DeclareMathOperator{\KerInt}{KerInt}
\DeclareMathOperator{\invlim}{\varprojlim}
\DeclareMathOperator{\Ker}{Ker}

\begin{document}

\newtheorem{thm}{Theorem}[section]
\newtheorem{cor}[thm]{Corollary}
\newtheorem{lem}[thm]{Lemma}
\newtheorem{prop}[thm]{Proposition}
\newtheorem{defin}[thm]{Definition}
\newtheorem{exam}[thm]{Example}
\newtheorem*{examples}{Examples}
\newtheorem{rem}[thm]{Remark}
\newtheorem{case}{\sl Case}
\newtheorem{claim}{Claim}
\newtheorem{prt}{Part}
\newtheorem*{mainthm}{Main Theorem}
\newtheorem*{thmm}{Theorem}
\newtheorem{question}[thm]{Question}
\newtheorem*{notation}{Notation}
\swapnumbers
\newtheorem{rems}[thm]{Remarks}
\newtheorem*{acknowledgment}{Acknowledgment}
\newtheorem{questions}[thm]{Questions}
\numberwithin{equation}{section}

\newcommand{\ab}{\mathrm{ab}}
\newcommand{\cont}{\mathrm{cont}}
\newcommand{\dirlim}{\varinjlim}
\newcommand{\discup}{\ \ensuremath{\mathaccent\cdot\cup}}
\newcommand{\divis}{\mathrm{div}}
\newcommand{\nek}{,\ldots,}
\newcommand{\inv}{^{-1}}
\newcommand{\isom}{\cong}
\newcommand{\pr}{\mathrm{pr}}
\newcommand{\sep}{\mathrm{sep}}
\newcommand{\tensor}{\otimes}
\newcommand{\alp}{\alpha}
\newcommand{\gam}{\gamma}
\newcommand{\Gam}{\Gamma}
\newcommand{\del}{\delta}
\newcommand{\Del}{\Delta}
\newcommand{\eps}{\epsilon}
\newcommand{\lam}{\lambda}
\newcommand{\Lam}{\Lambda}
\newcommand{\sig}{\sigma}
\newcommand{\Sig}{\Sigma}
\newcommand{\bfA}{\mathbf{A}}
\newcommand{\bfB}{\mathbf{B}}
\newcommand{\bfC}{\mathbf{C}}
\newcommand{\bfF}{\mathbf{F}}
\newcommand{\bfP}{\mathbf{P}}
\newcommand{\bfQ}{\mathbf{Q}}
\newcommand{\bfR}{\mathbf{R}}
\newcommand{\bfS}{\mathbf{S}}
\newcommand{\bfT}{\mathbf{T}}
\newcommand{\bfZ}{\mathbf{Z}}
\newcommand{\dbA}{\mathbb{A}}
\newcommand{\dbC}{\mathbb{C}}
\newcommand{\dbF}{\mathbb{F}}
\newcommand{\dbN}{\mathbb{N}}
\newcommand{\dbQ}{\mathbb{Q}}
\newcommand{\dbR}{\mathbb{R}}
\newcommand{\dbU}{\mathbb{U}}
\newcommand{\dbZ}{\mathbb{Z}}
\newcommand{\grf}{\mathfrak{f}}
\newcommand{\gra}{\mathfrak{a}}
\newcommand{\grm}{\mathfrak{m}}
\newcommand{\grp}{\mathfrak{p}}
\newcommand{\grq}{\mathfrak{q}}
\newcommand{\calA}{\mathcal{A}}
\newcommand{\calB}{\mathcal{B}}
\newcommand{\calC}{\mathcal{C}}
\newcommand{\calE}{\mathcal{E}}
\newcommand{\calG}{\mathcal{G}}
\newcommand{\calH}{\mathcal{H}}
\newcommand{\calJ}{\mathcal{J}}
\newcommand{\calK}{\mathcal{K}}
\newcommand{\calL}{\mathcal{L}}
\newcommand{\calW}{\mathcal{W}}
\newcommand{\calV}{\mathcal{V}}

\title{Filtrations of free groups as intersections}

\author{Ido Efrat}
\address{Mathematics Department\\
Ben-Gurion University of the Negev\\
P.O.\ Box 653, Be'er-Sheva 84105\\
Israel} \email{efrat@math.bgu.ac.il}

\thanks{This work was supported by the Israel Science Foundation (grant No.\ 152/13).}



\maketitle

\begin{abstract}
For several natural filtrations of a free group $S$ we express the $n$-th term of the filtration
as the intersection of all kernels of homomorphisms from $S$ to certain groups of upper-triangular unipotent matrices.
This generalizes a classical result of Gr\"un for the lower central filtration.
In particular, we do this for the $n$-th term  in the lower $p$-central filtration of $S$.
\end{abstract}

\section{Introduction}
We consider a group $G$ and decreasing filtrations $G_i$, $i=1,2\nek$ of $G$ by normal subgroups such that $G_1=G$ and $[G_i,G_j]\leq G_{i+j}$ for every $i,j\geq1$.
The \textbf{lower central filtration} $G^{(i)}$ is the fastest such filtration.
Next to it, for a prime number $p$, we have the \textbf{$p$-Zassenhaus filtration} $G_{(i,p)}$
and the \textbf{lower $p$-central filtration} $G^{(i,p)}$, which are the fastest filtrations as above
such that, in addition, $G_i^p\leq G_{ip}$, resp., $G_i^p\leq G_{i+1}$, for every $i\geq1$.
Here, for subgroups $H,K$ of $G$ we write as usual $[H,K]$ (resp., $H^p$, $HK$) for the subgroup generated
by all elements $[h,k]=h\inv k\inv hk$ (resp., $h^p$, $hk$) with $h\in H$ and $k\in K$.
More concretely, we define inductively:
\begin{enumerate}
\item[(i)]
$G^{(1)}=G$, \  $G^{(i)}=[G,G^{(i-1)}]$  for $i\geq2$;
\item[(ii)]
$G_{(1,p)}=G$, \ $G_{(i,p)}=(G_{(\lceil i/p\rceil,p)})^p\prod_{j+l=i}[G_{(j,p)},G_{(l,p)}]$ for $i\geq2$;
\item[(iii)]
$G^{(1,p)}=G$, \ $G^{(i,p)}=(G^{(i-1,p)})^p[G,G^{(i-1,p)}]$ for $i\geq2$.
\end{enumerate}
(See \cite{DixonDuSautoyMannSegal99}*{p.\ 5 and Prop.\ 1.16(2)} for the condition about the commutators in (i) and (iii)).
The $p$-Zassenhaus filtration is also called the \textbf{$p$-modular dimension filtration} \cite{DixonDuSautoyMannSegal99}*{Ch.\ XI}.

These filtrations also have their natural profinite analogs, where the subgroups are taken to be closed.

When $G=S$ is a free group, the subgroups $S^{(n)}$ (in the discrete case) and $S_{(n,p)}$ (in the profinite case) have known alternative descriptions in terms of linear representations.
Namely for a unital commutative ring $R$ let $\dbU_n(R)$ be the group of
all upper-triangular unipotent $n\times n$ matrices over $R$.
Then:

(1) \quad
When $S$ is a free discrete group on finitely many generators, $S^{(n)}=\bigcap\Ker(\varphi)$, where the intersection is over all group homomorphisms $\varphi\colon S\to \dbU_n(\dbZ)$ (Gr\"un \cite{Grun36}).

\medskip

(2) \quad
When $S$ is a free profinite group, $S_{(n,p)}=\bigcap\Ker(\varphi)$, where the intersection is over all continuous group homomorphisms $\varphi\colon S\to \dbU_n(\dbF_p)$
(this is a special case of \cite{Efrat14}, which is proved under a more general cohomological assumption on $n$-fold Massey products).

In this note we prove a general intersection theorem for free groups (Theorem \ref{theorem qqq}) which gives similar results in a variety of situations, both in the discrete and the profinite case, including (1) and (2) above.
 Moreover, it gives as a special case an analogous intersection theorem for the filtration $S^{(n,p)}$.
Namely, let $G(n,p)$ be  the group of all upper-triangular unipotent $n\times n$ matrices $(a_{ij})$
over $\dbZ/p^n\dbZ$ such that $a_{ij}\in p^{j-i}\dbZ/p^n\dbZ$ for every $i\leq j$.

\begin{thmm}
One has $S^{(n,p)}=\bigcap\Ker(\varphi)$,
where the intersection is over all group homomorphisms $\varphi\colon S\to G(n,p)$.
\end{thmm}
This result holds in both the discrete and profinite settings.
The main tool we use is the Magnus representation of $S$ by formal power series,
and a description of $S^{(n,p)}$ by means of this representation, due to Koch \cite{Koch60}.

The motivation for this work comes from Galois theory, where results of this nature
were studied for absolute Galois groups of fields.
More specifically, let $F$ be a field containing a root of unity of order $p$ and let $G=G_F$ be its absolute Galois group with the (profinite) Krull topology.
For a list $\calL$ of finite groups let $G_\calL=\bigcap\Ker(\varphi)$, where the intersection is over all
continuous epimorphisms $\varphi\colon G\to\bar G$ where $\bar G\in\calL$.
Let $D_4$ be the dihedral group of order $8$, and for $p$ odd let $H_{p^3}$ (resp., $M_{p^3}$)
be the unique non-abelian group of order $p^3$ and exponent $p$ (resp., $p^2$).
The following facts for $n=3$ were proved by Min\'a\v c, Spira, and the author:
\begin{enumerate}
\item[(i)]
$G^{(3,2)}=G_{(3,2)}=G_\calL$ for $\calL=\{1,\dbZ/2\dbZ, \dbZ/4\dbZ,D_4\}$ \cite{MinacSpira96} (see also \cite{EfratMinac13}*{Remark 2.1(1)});
\item[(ii)]
For $p>2$, $G_{(3,p)}=G_\calL$, where  $\calL=\{1,\dbZ/p\dbZ,H_{p^3}\}$ \cite{EfratMinac13};
\item[(iii)]
For $p>2$, $G^{(3,p)}=G_\calL$, where  $\calL=\{1,\dbZ/p^2\dbZ,M_{p^3}\}$ \cite{EfratMinac11}.
\end{enumerate}
The result (2) above extends (i) and (ii) to higher subgroups $G_{(n,p)}$ in the $p$-Zassenhaus filtration, at least in the case of free pro-$p$ groups.
Our Theorem solves the same problem for the lower $p$-central filtration $G^{(n,p)}$.

The paper is organized as follows:
In \S2 we recall some basic facts on free pro-$\calC$ groups on a (possibly infinite) basis $A$, where $\calC$ is a full formation of finite groups.
In \S3 and \S4 we extend the classical construction of the Magnus algebra $R_0\langle\langle X_A\rangle\rangle\rangle$ 
and Magnus homomorphism $\Lam_{R_0,A}\colon S\to R_0\langle\langle X_A\rangle\rangle\rangle$ from the discrete setting to the setting of free pro-$\calC$ groups,
where $R_0$ is a unital ring.
The heart of the proof is given in \S5.
There we consider a ring homomorphism $\theta\colon R_0\to R$ and a system $\calJ=(J_k)_{k=0}^{n-1}$ of ideals in $R$ satisfying some natural conditions.
This gives rise to the ring $T_{n,0}(\calJ)$ of upper-triangular $n\times n$-matrices whose $(i,j)$-entries lie in $J_{j-i}$, and to the group $\dbU_n(\calJ)$ of unipotent matrices in $T_{n,0}(\calJ)$.
Every representation $\varphi\colon S\to\dbU_n(\calJ)$ lifts in a canonical way to a ring homomorphism $\hat\varphi\colon R_0\langle\langle X_A\rangle\rangle\to T_{n,0}(\calJ)$ which is $\theta I_n$ on $R_0$
and such that the following square commutes (see Lemma \ref{sss}):
\begin{equation}
\label{com square}
\xymatrix{
S\ar[r]^{\varphi}\ar[d]_{\Lam_{R_0,A}} & \dbU_n(\calJ)\ar@{_(->}[d] \\
R_0\langle\langle X_A\rangle\rangle \ar^{\hat\varphi}[r] & T_{n,0}(\calJ).
}
\end{equation}
Under the additional assumption that $J_t=d^tR$ for some $d\in R$, $t=0,1\nek n-1$, we use this lifting to identify the Magnus expansions 
of the elements of $\bigcap\Ker(\varphi)$, where $\varphi$ ranges over all such representations (Theorem \ref{theorem qqq}).
Finally, in \S6 we apply this general result to obtain intersection theorems in all the above-mentioned special cases.

I warmly thank J\'an Min\'a\v c for many discussions that motivated this work.
I also thank the referee for his helpful comments and suggestions.

After an earlier version of this paper was posted on the arXiv, J\'an Min\'a\v c and Nguyen Duy Tan sent me
an advanced draft of their preprint \cite{MinacTan2}, where they prove a result in the spirit of the Theorem above,
but with the target group $G(n,p)$ replaced by the groups $\dbU_{k+1}(\dbZ/p^{n-k}\dbZ)$, $k=1\nek n - 1$.
They also recover Gr\"un's result (1).
Their general approach was to use the Magnus theory in combination with \cite{Koch02}*{Th.\ 7.14}, which is the group ring analog of \cite{Koch60}.

\section{Free profinite groups}
Let $\calC$ be a \textbf{full formation} of finite groups, i.e., a non-empty family of finite groups closed under subgroups, epimorphic images and extensions (in the sense that if $N$ is a normal subgroup of a group $G$ and $N,G/N\in\calC$, then $G\in\calC$; see \cite{FriedJarden08}*{\S17.3}).
We recall from \cite{FriedJarden08}*{\S17.4} the following terminology and facts about pro-$\calC$ groups.

Let $G$ be a pro-$\calC$ group and $A$ a set.
A map $\varphi\colon A\to G$  \textbf{converges to $1$} if for every open normal subgroup $N$ of $G$,
the set $A\setminus\varphi\inv(N)$ is finite.

We say that a pro-$\calC$ group $S$ is a \textbf{free pro-$\calC$ group on basis $A$} with respect to a map $\iota\colon A\to S$ if
\begin{enumerate}
\item[(i)]
$\iota\colon A\to S$ converges to $1$;
\item[(ii)]
$\iota(A)$ generates $S$;
\item[(iii)]
For every pro-$\calC$ group $G$ and every map $\varphi\colon A\to G$ converging to $1$, there
is a unique continuous homomorphism $\hat\varphi\colon S\to G$ with
$\varphi=\hat\varphi\circ\iota$ on $A$.
\end{enumerate}
A free pro-$\calC$ group on $A$ exists, and is unique up to a unique continuous isomorphism.
We denote it  by  $S_A(\calC)$.
Necessarily, $\iota$ is injective, and we identify $A$ with its image in $S_A(\calC)$.
In particular, we write $\hat\dbZ_\calC$ for the free pro-$\calC$ group on one generator.

\section{The Magnus Algebra}
Let $A$ be a set and $A^*$ the set of all finite sequences of elements of $A$.
We denote the empty sequence by $\emptyset$.
For each $a\in A$ let $X_a$ be a variable.
For $I=(a_1\nek a_t)\in A^*$ let $X_I$ denote the formal product $X_{a_1}\cdots X_{a_t}$ (by convention $X_\emptyset=1$).
Let $R$ be a unital commutative ring with additive group $R_+$.
The \textbf{Magnus algebra} $R\langle\langle X_A\rangle\rangle$
over  $R$ is the ring of all formal power series
$\sum_{I\in A^*}c_IX_I$, where $c_I\in R$, with the natural operations.
We identify $R$ as the subring of $R\langle\langle X_A\rangle\rangle$ of all constant power series.
Let $R\langle\langle X_A\rangle\rangle^\times$ be the multiplicative group of $R\langle\langle X_A\rangle\rangle$, and for a positive integer $n$ let $V_{R,A,n}$ be the set of all power series in $R\langle\langle X_A\rangle\rangle$ of the form $1+\sum_{|I|\geq n}c_IX_I$.
For $1+\alp\in V_{R,A,n}$ we have $(1+\alp)\inv=\sum_{k=0}^\infty(-1)^k\alp^k$, showing that $V_{R,A,n}$ is in fact a normal subgroup of $V_{R,A,1}$.
We have
\begin{equation}
\label{invlim}
V_{R,A,1}\isom\invlim V_{R,A,1}/V_{R,A,n}.
\end{equation}

\begin{lem}
\label{direct product}
For every positive integer $n$ there is a group isomorphism
\[
\Psi_n\colon
\prod_{|I|=n}R_+\to V_{R,A,n}/V_{R,A,n+1},  \quad
(c_I)_{|I|=n}\mapsto(1+\sum_{|I|=n}c_IX_I)V_{R,A,n+1}.
\]
\end{lem}
\begin{proof}
It is straightforward to see that $\Psi_n$ is a group homomorphism and is injective.
For the surjectivity, let $1+\sum_{|I|\geq n}c_IX_I\in V_{R,A,n}$.
Then
\[
\begin{split}
(1+\sum_{|I|\geq n}c_IX_I)(1+\sum_{|I|=n}c_IX_I)\inv
&=1+(\sum_{|I|\geq n+1}c_IX_I)(1+\sum_{|I|=n}c_IX_I)\inv \\
&\in 1+(\sum_{|I|\geq n+1}c_IX_I)V_{R,A,n}\subseteq V_{R,A,n+1}.
\end{split}
\]
Hence $\Psi((c_I)_{|I|=n})=(1+\sum_{|I|\geq n}c_IX_I)V_{R,A,n+1}$.
\end{proof}

The map $\sum_Ic_IX_I\mapsto(c_I)_I$ identifies $R\langle\langle X_A\rangle\rangle$ with $\prod_{I\in A^*}R$.
When $R$ is a profinite topological ring, this induces on the additive group of $R\langle\langle X_A\rangle\rangle$ a profinite (product) topology.
Moreover, the multiplication map in  $R\langle\langle X_A\rangle\rangle$ is continuous,
making it a profinite topological ring.
Then the isomorphisms (\ref{invlim}) and $\Psi_n$ are continuous.

\section{The Magnus homomorphism}
Here we separate the discussion into the discrete case and the profinite case.

Suppose that $A$ is a finite set and $S$ is a free discrete group on the basis $A$.
Let $R$ be a (discrete) unital commutative ring.
We define the \textbf{Magnus homomorphism}
$\Lam=\Lam_{R,A}\colon S\to R\langle\langle X_A\rangle\rangle^\times$
 by $\Lam(a)=1+X_a$ for $a\in A$.
It is known to be injective (\cite{Magnus35}*{Satz I}, \cite{SerreCG}*{I-\S1.4}).

Next we define the Magnus homomorphism in a profinite setting.
Let $\Pi$ be a set of prime numbers and let $\calC=\calC(\Pi)$ be the family of all finite $\Pi$-groups, i.e., finite groups whose order is a product of primes in $\Pi$.
It is a full formation.
Let $R=\hat \dbZ_\calC=\prod_{p\in \Pi}\dbZ_p$ and note that it is a profinite ring.

\begin{lem}
\label{pro-C}
In this setup, $V_{R,A,1}$ is a pro-$\calC$ group.
\end{lem}
\begin{proof}
By Lemma \ref{direct product}, $V_{R,A,n}/V_{R,A,n+1}$ is a pro-$\calC$ group for every $n$.
Using the extension
\[
1\to V_{R,A,n}/V_{R,A,n+1}\to V_{R,A,1}/V_{R,A,n+1}\to V_{R,A,1}/V_{R,A,n}\to 1
\]
we conclude by induction that $V_{R,A,1}/V_{R,A,n}$ is a pro-$\calC$ group for every $n$.
Now use the isomorphism (\ref{invlim}).
\end{proof}

For a finite subset $B$ of $A$,  $V_{R,A\setminus B,1}$ is a closed subgroup of $V_{R,A,1}$.
We have $\bigcap_BV_{R,A\setminus B,1}=\{1\}$.
Hence every open normal subgroup of $V_{R,A,1}$ contains some $V_{R,A\setminus B,1}$ \cite{RibesZalesskii10}*{Prop.\ 2.1.5(a)}.
If $a\in A$ and $1+X_a\not\in V_{R,A\setminus B,1}$, then $a\in B$.
Therefore the map $A\to V_{R,A,1}$, $a\mapsto 1+X_a$, converges to $1$.
In view of Lemma \ref{pro-C}, it extends to the continuous \textbf{Magnus homomorphism}
$\Lam_{\calC,A}\colon S_A(\calC)\to V_{R,A,1}\leq R\langle\langle X_A\rangle\rangle^\times$.

\section{Unipotent representations}
We  fix $n\geq1$ and a unitary commutative ring $R$.
We write $I_n$ for the $n\times n$ identity matrix.

Let $\calJ=(J_k)_{k=0}^{n-1}$ be a sequence of ideals in $R$ such that
$R=J_0\supseteq J_1\supseteq J_2\supseteq\cdots\supseteq J_{n-1}$ and $J_kJ_l\subseteq J_{k+l}$ for every $k,l\geq0$ with $k+l\leq n-1$.
Given a non-negative integer $t$, let $T_{n,t}(\calJ)$ be the set of all  $n\times n$ matrices $(a_{ij})$ over $R$
such that
\begin{enumerate}
\item[(i)]
$a_{ij}=0$ for every $1\leq i,j\leq n$ such that $j-i\leq t-1$;
\item[(ii)]
$a_{ij}\in J_{j-i}$ for every $1\leq i\leq j\leq n$.
\end{enumerate}

\begin{rems}
\label{rems on T}
\rm
\begin{enumerate}
\item[(1)]
$T_{n,0}(\calJ)$ is a unital $R$-algebra with respect to the usual operations.
\item[(2)]
$T_{n,t}(\calJ)=\{0\}$ for $n\leq t$.
\item[(3)]
$\dbU_n(\calJ):=T_{n,0}(\calJ)\cap\dbU_n(R)=I_n+T_{n,1}(\calJ)$ is a multiplicative group.
\item[(4)]
The entries of matrices in $T_{n,t}(\calJ)$ are in $J_t$ when $0\leq t<n$.
\item[(5)]
$T_{n,t}(\calJ)T_{n,t'}(\calJ)\subseteq T_{n,t+t'}(\calJ)$ for $t,t'\geq 0$.
\end{enumerate}
\end{rems}

Now let $\theta\colon R_0\to R$ be a homomorphism of unital commutative rings.
Let $A$ be a finite set, and let $S$ be the free discrete group on basis $A$.

\begin{lem}
\label{sss}
Let $\varphi\colon S\to \dbU_n(\calJ)$ be a group homomorphism.
For every sequence $I=(a_1\nek a_t)\in A^*$ let $M_I=\prod_{k=1}^t(\varphi(a_k)-I_n)$
 (by convention $M_\emptyset=I_n$).
Then:
\begin{enumerate}
\item[(a)]
 $M_I\in T_{n,t}(\calJ)$.
\item[(b)]
There is a homomorphism of unital rings
\[
 \hat\varphi\colon R_0\langle\langle X_A\rangle\rangle\to T_{n,0}(\calJ), \quad
\sum_Ic_IX_I\mapsto\sum_{0\leq|I|<n}\theta(c_I)M_I.
\]
It satisfies $\hat\varphi(c)=\theta(c)I_n$ for $c\in R_0$, and $\varphi=\hat\varphi\circ\Lam_{R_0,A}$ on $S$,
and is the unique homomorphism with these properties (see diagram (\ref{com square})).
\end{enumerate}
\end{lem}
\begin{proof}
(a) \quad
The matrices $\varphi(X_{a_1})-I_n\nek\varphi(X_{a_t})-I_n$ are in $T_{n,1}(\calJ)$.
By Remark \ref{rems on T}(5),  $M_I\in T_{n,1}(\calJ)^t\subseteq T_{n,t}(\calJ)$.

\medskip

(b)\quad
By (a) and Remark \ref{rems on T}(2),  $M_I=0$ when $|I|\geq n$.
It follows that  $\hat\varphi$ is a homomorphism.
By definition, $\hat\varphi(c)=\hat\varphi(cX_\emptyset)=\theta(c)M_\emptyset=\theta(c)I_n$, and for every generator $a\in A$ we have
\[
\varphi(a)=I_n+M_{(a)}=\hat\varphi(1+X_a)=(\hat\varphi\circ\Lam_{R_0,A})(a),
\] 
so $\varphi=\hat\varphi\circ\Lam_{R_0,A}$ on $S$.

Further, if $\hat\varphi'\colon R_0\langle\langle X_A\rangle\rangle\to T_{n,0}(\calJ)$ is another unital ring homomorphism satisfying the above properties, 
then $\hat\varphi'(c)=\theta(c)I_n=\hat\varphi(c)$ for every $c\in R_0$, and
$\hat\varphi'(X_a)=\hat\varphi'(\Lam_{R_0,A}(a)-1)=\varphi(a)-I_n=\hat\varphi(X_a)$ for every $a\in A$.
It follows that $\hat\varphi'=\hat\varphi$.
\end{proof}

Next let $L_\calJ$ be the set of all power series $\sum_Ic_IX_I$ in
$R_0\langle\langle X_A\rangle\rangle$ such that $c_\emptyset=1$ and $\theta(c_I)\in \Ann_R(J_t)$
for every sequence $I$ of length $1\leq t<n$.
Here $\Ann_R(J_t)$ denotes the annihilator of $J_t$ in $R$.

\begin{lem}
\label{lemma 2}
$\varphi(\Lam_{R_0,A}\inv(L_\calJ))=I_n$ for every group homomorphism $\varphi\colon S\to \dbU_n(\calJ)$.
\end{lem}
\begin{proof}
Let $\varphi$ be a homomorphism as above, and let $M_I\in T_{n,t}(\calJ)$ and $\hat\varphi$ be as in Lemma \ref{sss}.
Consider a power series $\sum_Ic_IX_I\in L_\calJ$.
If $I\in A^*$ is a sequence of length $1\leq t<n$, then $\theta(c_I)\in \Ann_R(J_t)$, and the entries of $M_I$ are in $J_t$
(Remark \ref{rems on T}(4)).
Hence $\hat\varphi(c_IX_I)=\theta(c_I)M_I=0$.
Consequently, $\hat\varphi(\sum_Ic_IX_I)=c_\emptyset M_\emptyset=I_n$.

Thus if $\sig\in S$ and $\Lam_{R_0,A}(\sig)\in L_\calJ$, then $\varphi(\sig)=\hat\varphi(\Lam_{R_0,A}(\sig))=I_n$.
\end{proof}

We define the \textsl{kernel intersection} of groups $H,G$ to be
\[
\KerInt(H,G)=\bigcap\{\Ker(\varphi)\ |\ \varphi\colon H\to G \text{\rm \ homomorphism}\}.
\]

\begin{thm}
\label{theorem qqq}
Suppose that $J_t=d^tR$, $t=0,1\nek n-1$, for some $d\in R$.
Then
\[
\Lam_{R_0,A}\inv(L_\calJ)=\KerInt(S,\dbU_n(\calJ)).
\]
\end{thm}
\begin{proof}
Lemma \ref{lemma 2} gives the inclusion $\subseteq$.
For the opposite inclusion let $\sig\in\KerInt(S,\dbU_n(\calJ))$ and write
$\Lam_{R_0,A}(\sig)=\sum_Ic_IX_I$ (so $c_\emptyset=1$).
We need to show that $d^t\theta(c_{I_0})=0$ for every sequence $I_0=(l_1\nek l_t)\in A^*$ of length $1\leq t<n$.
We may assume inductively that $d^s\theta(c_I)=0$ for every sequence $I\in A^*$ of length $1\leq s<t$.

Let $E_{ij}$ be the $n\times n$ matrix over $R$ which is $1$ at entry $(i,j)$ and $0$ elsewhere.
We recall that $E_{ij}E_{i'j'}$ is $E_{ij'}$, if $i'=j$, and is $0$ otherwise.
Hence a product $E_{j_1,j_1+1}\cdots E_{j_s,j_s+1}$, where $1\leq j_1\nek j_s<n$, is non-zero if and only if $j_1\nek j_s$ are consecutive numbers,
and in this case it equals $E_{j_1,j_s+1}$.

For every $a\in A$ we define
\[
M_a=d\sum_jE_{j,j+1}\in T_{n,1}(\calJ),
\]
where the sum is over all $1\leq j\leq t$ such that $l_j=a$.
Since $S$ is free, there is a group homomorphism $\varphi\colon S\to\dbU_n(\calJ)$ such that $\varphi(a)=I_n+M_a$ for every $a\in A$.
Let $\hat\varphi\colon R_0\langle\langle X_A\rangle\rangle\to T_{n,0}(\calJ)$ be the ring homomorphism as in Lemma \ref{sss}(b).
By the assumption on $\sig$,
\[
0=\varphi(\sig)-I_n=\hat\varphi(\Lam_{R_0,A}(\sig)-1)=\sum \theta(c_I)M_I,
\]
where the sum is over all sequences $I=(a_1\nek a_s)\in A^*$ with $1\leq s<n$.

Given such a sequence $I$, the matrix $M_I=M_{a_1}\cdots M_{a_s}$ is the sum of all products
$d^sE_{j_1,j_1+1}\cdots E_{j_s,j_s+1}$ such that $1\leq j_1\nek j_s\leq t$ and  $l_{j_1}=a_1\nek l_{j_s}=a_s$.
We now break the computation into three cases:

\textsl{Case 1:}
 $1\leq s<t$. \quad
Then the induction hypothesis gives $d^s\theta(c_I)=0$, whence $d^s\theta(c_I)E_{j_1,j_1+1}\cdots E_{j_s,j_s+1}=0$.

\textsl{Case 2:}
 $s\geq t$, $(j_1\nek j_s)\neq(1\nek t)$.  \quad
Since $1\leq j_1\nek j_s\leq t$ this implies that $j_1\nek j_s$ are not consecutive numbers,
and therefore $E_{j_1,j_1+1}\cdots E_{j_s,j_s+1}=0$.

\textsl{Case 3:}
 $s=t$, $(j_1\nek j_t)=(1\nek t)$. \quad
Then $I=I_0$ and
\[
d^s\theta(c_I)E_{j_1,j_1+1}\cdots E_{j_s,j_s+1}=d^t\theta(c_{I_0})E_{1,t+1}.
\]

Altogether we obtain that
\[
0=\sum_{1\leq |I|<n}\theta(c_I)M_I=d^t\theta(c_{I_0})E_{1,t+1}.
\]
It follows that $d^t\theta(c_{I_0})=0$, as required.
\end{proof}

\begin{rems}
\label{remark on profinite analog}
\rm
(a) \quad
In the previous proof, if  $a$ does not appear in $I_0$, then $M_a=0$.

\medskip

(b) \quad
There is also a profinite analog of Theorem \ref{theorem qqq}:
Let $\Pi$ be a set of prime numbers and let $\calC$ be the formation of all finite $\Pi$-groups.
Let $R_0=\hat \dbZ_\calC$, let $R$ be a profinite ring whose additive group is pro-$\calC$, and let $\theta\colon R_0\to R$ be a continuous ring homomorphism.
We choose $d\in R$ and set $J_t=d^tR$ for $t=0,1\nek n-1$.
Then $\dbU_n(\calJ)$ is a pro-$\calC$ group.
Also let $S=S_A(\calC)$ be the free pro-$\calC$ group on basis $A$.
We note that $\Ann_R(J_t)$ is closed in $R$, so $L_\calJ$ is closed in the profinite ring $\dbZ_\calC\langle\langle X_A\rangle\rangle$.
Lemma \ref{sss}, Lemma \ref{lemma 2}, and Theorem \ref{theorem qqq} and their proofs hold almost without any changes, with homomorphisms understood to be continuous, with $\KerInt(S,\dbU_n(\calJ))$ replaced by
\[
\begin{split}
\KerInt_\cont&(S,\dbU_n(\calJ)) \\
&=\bigcap\{\Ker(\varphi)\ |\ \varphi\colon S\to\dbU_n(\calJ) \hbox{ continuous homomorphism}\},
\end{split}
\]
and using Remark (a) to see that the map $A\to \dbU_n(\calJ)$, $a\mapsto I_n+M_a$ converges to $1$.
\end{rems}

\section{Examples}

\begin{exam}
\label{exam 1}
\rm
Suppose that $R=R_0$ and $\theta$ is the identity map.
Take in Theorem \ref{theorem qqq} $d=1$.
Then $J_t=R$  and $\Ann_R(J_t)=\{0\}$ for $0\leq t\leq n-1$.
Thus $L_\calJ$ consists of all power series $1+\sum_{|I|\geq n}c_IX_I$ with $c_I\in R$.
Moreover,  here $\dbU_n(\calJ)=\dbU_n(R)$.
For a free discrete group $S$ on a finite set $A$ of generators Theorem \ref{theorem qqq} gives
\begin{equation}
\label{ttt}
\Lam_{R,A}\inv(L_J)=\KerInt(S,\dbU_n(R)).
\end{equation}
\end{exam}

\begin{exam}
\label{exam 2}
\rm
Take in Example \ref{exam 1} $R=R_0=\dbZ$.
As proved by Magnus \cite{Magnus37}*{Satz III}, Witt \cite{Witt37}, and Gr\"un \cite{Grun36}
(see \cite{SerreLie}*{Ch.\ IV, Th.\ 6.3} for a more modern approach),
  $\Lam_{\dbZ,A}\inv(L_\calJ)$ is the $n$-th term $S^{(n)}$ in the lower central filtration of $S$.
We deduce from (\ref{ttt}) that
\[
S^{(n)}=\KerInt(S,\dbU_n(\dbZ)).
 \]
This was proved in  \cite{Grun36};
see also the modern exposition of Gr\"un's work in \cite{Rohl85}, as well as the related work \cite{Magnus35}*{Satz VI}.
We remark that Gr\"un actually works with {\sl lower}-triangular unipotent matrices.
\end{exam}

\begin{exam}
\label{exam 3}
\rm
Let $\Pi$ be a set of prime numbers and let $\calC=\calC(\Pi)$ be the formation of all finite $\Pi$-groups.
Let $A$ be a set, and $S=S_A(\calC)$ the free pro-$\calC$ group on basis $A$.
We similarly obtain, in view of Remark \ref{remark on profinite analog}, that
\[
S^{(n)}=\KerInt_\cont(S,\dbU_n(\hat\dbZ_\calC)).
 \]
In particular, for a free profinite group $S$ we have
$S^{(n)}=\KerInt(S,\dbU_n(\hat\dbZ))$, and for a free pro-$p$ group $S$ we have $S^{(n)}=\KerInt(S,\dbU_n(\dbZ_p))$.
\end{exam}

\begin{exam}
\label{exam 4}
\rm
Let $R=R_0=\dbF_p$ and $S$ a free discrete group on a finite set $A$ of generators.
Then  $\Lam_{\dbF_p,A}\inv(L_\calJ)$ is the $n$-th term $S_{(n,p)}$ in the $p$-Zassenhaus filtration of $S$ (compare \cite{Vogel05}*{Lemma 2.19(ii)}, \cite{Morishita12}*{\S8.3}, \cite{Efrat14}*{Prop.\ 6.2}, and \cite{Jennings41}*{Th.\ 5.5}).
We conclude from (\ref{ttt}) that
\[
S_{(n,p)}=\KerInt(S,\dbU_n(\dbF_p)).
\]

Again, the same result holds when $S=S_\calC(A)$ is a free pro-$\calC$ group on a basis $A$,
with $\calC$ the formation of finite $\Pi$-groups for some set $\Pi$ of prime numbers.
In fact, this was proved in \cite{Efrat14}*{Th.\ A'} for any profinite group $S$
with $p$-cohomological dimension $\leq1$,  as a special case of a deeper cohomological result,
related to Massey products.

In the special case where $n=3$,  we note that $\dbU_3(\dbF_2)=D_4$, and $\dbU_3(\dbF_p)=H_{p^3}$ when $p>2$.
Hence the list $\calL$ of all subgroups of  $\dbU_3(\dbF_p)$ consists of
$\{1\},\dbZ/2\dbZ,\dbZ/4\dbZ,D_4$, when $p=2$, and $\{1\},\dbZ/p\dbZ,(\dbZ/p\dbZ)^2,H_{p^3}$, when $p>2$.
For a free profinite group $S$ we deduce that $S_{(3,p)}=\KerInt(S,\dbU_3(\dbF_p))=S_\calL$ (with notation as in the Introduction).
Furthermore, when $p>2$ we may replace here $\calL$ by $\calL'=\{\{1\},\dbZ/p\dbZ,H_{p^3}\}$.
This recovers the results of \cite{MinacSpira96} and \cite{EfratMinac13} mentioned in the Introduction (for $G=S$).
\end{exam}

\begin{exam}
\label{exam 5}
\rm
Let $S$ be a \textsl{free pro-$p$ group}.
We recall that $\dbU_n(\dbF_p)$ is a $p$-Sylow subgroup of $\GL_n(\dbF_p)$ (see e.g., \cite{RibesZalesskii10}*{Ex.\ 2.3.12}).
It follows from Example \ref{exam 4} that
\[
S_{(n,p)}=\KerInt(S,\GL_n(\dbF_p)).
\]
\end{exam}

\begin{exam}
\label{exam 6}
\rm
Let $p$ be a prime number, $R_0=\dbZ$,  and $R=\dbZ/p^n\dbZ$, and $\theta\colon\dbZ\to\dbZ/p^n\dbZ$ the natural epimorphism.
We take in Theorem \ref{theorem qqq} $d=p+p^n\dbZ\in R$, so $J_t=p^t\dbZ/p^n\dbZ\subseteq R$ and $\Ann_R(J_t)=p^{n-t}\dbZ/p^n\dbZ$ for $0\leq t\leq n-1$.
Thus, in the terminology of the Introduction, $\dbU_n(\calJ)=G(n,p)$, and $L_\calJ$ consists of all power series $1+\sum_Ic_IX_I$ 
in $\dbZ\langle\langle X_A\rangle\rangle$ such that $c_I\in p^{n-|I|}\dbZ$ for $1\leq|I|<n$,
and $c_I\in\dbZ$ for $|I|\geq n$.

Let $S$ be a discrete free group on a finite set $A$ of generators.
Let $D$ be the two-sided ideal in $\dbZ\langle\langle X_A\rangle\rangle$ generated by $X_a$, $a\in A$.
Then $L_\calJ=1+\sum_{1\leq t\leq n}p^{n-t}D^t$, where $D^t$ is the ideal of all sums of products of $t$ elements
of $D$.
We observe that  $L_\calJ=1+(p\dbZ+D)^n$.
By a theorem of Koch \cite{Koch60},  $\Lam_{\dbZ,A}\inv(1+(p\dbZ+D)^n)$ is the $n$-th term
$S^{(n,p)}$ in the lower $p$-central series of $S$.
Combining this with Theorem \ref{theorem qqq}, we obtain the Theorem from the Introduction:
\[
S^{(n,p)}=\KerInt(S,G(n,p)).
\]
A similar result holds for a free pro-$\calC$ group $S=S_A(\calC)$, with $\calC$ the formation of all finite $\Pi$-groups for set $\Pi$ of prime numbers (of course, here $S^{(n,p)}$ is in the profinite sense):
\[
S^{(n,p)}=\KerInt_\cont(S,G(n,p)).
\]
\end{exam}

We list some consequences of these examples (the first three facts seem to be well-known):

\begin{cor}
\label{triviality of subgroups}
\begin{enumerate}
\item[(a)]
In the discrete setting, $\dbU_n(\dbZ)^{(n)}=1$.
\item[(b)]
In the profinite setting, $\dbU_n(\dbZ_\calC)^{(n)}=1$.
\item[(c)]
$\dbU_n(\dbF_p)_{(n,p)}=1$.
\item[(d)]
$G(n,p)^{(n,p)}=1$.
\end{enumerate}
\end{cor}
\begin{proof}
We prove (d).
Take a free group $S$ on sufficiently many generators and an epimorphism $\varphi\colon S\to G(n,p)$.
By the definition of the lower $p$-central filtration, it maps $S^{(n,p)}$ onto $G(n,p)^{(n,p)}$.
But the Theorem implies that $\varphi$ is trivial on $S^{(n,p)}$.

(a)--(c) are proved similarly, taking $S$ to be a free group in the relevant context, and using Examples \ref{exam 2},
\ref{exam 3}, and \ref{exam 4}, respectively.
\end{proof}

\end{document}